\documentclass[paper=a4]{scrartcl}
\usepackage{graphicx}
\usepackage{amsmath,amsfonts,amssymb,amsthm}
\usepackage{braket}
\usepackage{csquotes}
\usepackage{ifluatex}
\ifluatex
\usepackage{polyglossia}
\setdefaultlanguage{english}
\usepackage[colon=literal]{unicode-math}
\AtBeginDocument{\directlua{require("combining.lua")}}
\newcommand{{\hat \}}{\hat}
\newcommand{{\tilde \}}{\tilde}
\newcommand{{\overline \}}{\bar}
\else
\usepackage[english]{babel}
\usepackage[utf8]{inputenc}
\newcommand{\eps}{\varepsilon}
\usepackage{mathtools}
\makeatletter
\InputIfFileExists{\jobname.pre2}{}{}
\makeatother
\fi
\usepackage[unicode]{hyperref}
\usepackage[hypcap=true]{caption}
\usepackage[hypcap=true]{subcaption}

\usepackage[inline]{asymptote}

\begin{asydef}
    import lib;
    import graph_settings;
    defaultpen(fontsize(10));
    unitsize (15mm);
\end{asydef}

\usepackage{scrlayer-scrpage}
\pagestyle{scrheadings}
\ohead{\headmark}
\ihead{New structurally unstable families}
\chead{}
\automark{section}
\cfoot{\pagemark}

\usepackage[giveninits=true,maxnames=10,backend=biber,sortcites]{biblatex}
\addbibresource{math.bib}
\usepackage[inline]{enumitem}
\setlist[enumerate]{noitemsep}
\setlist[itemize]{noitemsep}

\newtheorem{theorem}{Theorem} 
\newtheorem{lemma}{Lemma} 
 
\newtheorem{corollary}{Corollary} 
 
\theoremstyle{definition}
\newtheorem{definition}{Definition} 
\theoremstyle{remark}
\newtheorem{remark}{Remark} 

\DeclareMathOperator{\Vect}{Vect}
\DeclareMathOperator{\codim}{codim}
\DeclareMathOperator{\const}{const}

\AtBeginDocument{

}

\author{Nataliya Goncharuk \and Yury G. Kudryashov \and Nikita Solodovnikov}
\title{New structurally unstable families of planar vector fields}
\hypersetup{
 pdfauthor={Nataliya Goncharuk, Yury G. Kudryashov, Nikita Solodovnikov},
 pdftitle={New structurally unstable families of planar vector fields},
 pdfkeywords={bifurcations, vector fields},
 pdfsubject={},
 pdflang={English}}

\begin{document}

\maketitle

\begin{abstract}
    We study global bifurcations in generic \(3\)-parameter families of vector fields on \(S^2\).
    In the recent article~\parentext{\citeauthor{IKS-th1}, \citeyear{IKS-th1}}, the authors show that 3-parameter unfoldings of vector fields with the polycycle \enquote{tears of the heart} are structurally unstable.
    We consider 3-parameter unfoldings of vector fields with separatrix graphs \enquote{ears} and \enquote{glasses}, and prove that these families are structurally unstable as well.

    We also study in more details the classical bifurcation of a saddle loop, and use it as a building block in our main example.
\end{abstract}

\section{Introduction}%
\label{sec:intro}
It is well-known~\cite{AP37,Ba52,P59,P62,P63,ALGM66:en,ALGM67:en} that a generic vector field on~\(S^2\) is structurally stable.
\citeauthor{S74}~\cite{S74} classified so-called \emph{quasi-generic} vector fields on the sphere, i.e.\ vector fields generic inside the class of structurally unstable vector fields.
Quasi-generic vector fields have exactly one of the following degeneracies:
\begin{enumerate}
  \item a quasi-generic singular point, i.e.\ a saddle-node, or a composed focus;
  \item a separatrix connection between different saddles;
  \item\label{it:parabolic} a semi-stable limit cycle of multiplicity~\(2\);
  \item\label{it:sep-loop} a separatrix loop.
\end{enumerate}
He also described bifurcations that occur in \(1\)-parametric unfoldings of quasi-generic vector fields \emph{near the \enquote{interesting} parts of their phase portraits}.
In particular, this description implies that generic \(1\)-parameter unfoldings are structurally stable, if one restricts them to small neighbourhoods of their degeneracies listed above.
Bifurcations of this type are called \emph{local}, if the degeneracy is a singular point, and \emph{non-local}, if it is a polycycle or a degenerate limit cycle.

However, the corresponding \emph{global} bifurcations can be more complicated.
In~\cite{MP}, \citeauthor{MP} described global bifurcations in the two most interesting cases~\ref{it:parabolic} and~\ref{it:sep-loop}, stated their classification up to the topological equivalence (this and other equivalence relations are defined in~\autoref{def:equiv}), and sketched the proof of this classification.
In case~\ref{it:parabolic} of a semi-stable limit cycle this classification has numerical moduli, and in case~\ref{it:sep-loop} of a separatrix loop this classification has no moduli.
Later \citeauthor{DR90}~\cite{DR90} proved that in the latter case numerical moduli appear, if one imposes additional regularity assumptions on the conjugating homeomorphism.
In \autoref{sec:saddle-loop} we prove a modified version of this result for the \emph{weak topological equivalence} with similar modifications.

For families with more parameters, classification up to the topological equivalence has functional moduli.
This was proved by \citeauthor{R87}~\cite{R87} for \emph{non-local bifurcations} in a \(3\)-parameter family, and this example can be easily modified to produce functional invariants for \emph{global bifurcations} in a \(2\)-parameter family, see~\cite{GI-equivs} for details and another way to obtain functional invariants in the latter case.

In 1986, Arnold~\cite{AAIS94} conjectured\footnote{According to Yu.~Ilyashenko, this section of the book was written entirely by V.~Arnold.} that these moduli, both numerical and functional, appear because the topological equivalence relation is too strict.
He proposed a much less restrictive notion of \emph{weak topological equivalence}, and formulated \(6\)~conjectures about global bifurcations of finite parameter families of vector fields on~\(S^2\).
One of them states that a generic family of this type is weakly structurally stable.
Though most of these conjectures turned out to be wrong, they influenced research in this area for decades.

\begin{figure}[h]
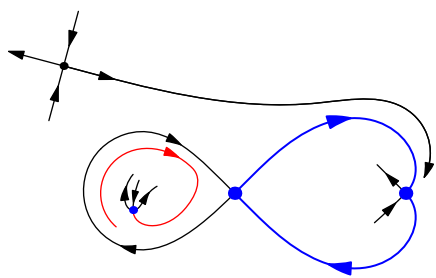

    \centering
    \asyinclude{TH}
    \caption{A vector field with polycycle \enquote{tears of the heart}}%
    \label{fig:TH}
\end{figure}

Recently \citeauthor{IKS-th1}~\cite{IKS-th1} came up with an open set in the space of \(3\)-parameter families of vector fields on the sphere such that all families in this set are structurally unstable.
These families are unfoldings of vector feilds with \enquote{tears of the heart} polycycle, see \autoref{fig:TH}.
For technical reasons, this result was proved for the \emph{moderate equivalence} relation that is stronger than the weak equivalence but is weaker than the original notion of topological equivalence.
Later, a modification of the construction used in~\cite{IKS-th1} was used in~\cite{GK-phase} to construct metrically generic \(3\)-parameter families of vector fields with arbitrafily many numerical invariants.

This result motivated a few series of questions about generic families of vector fields on the sphere.
\begin{description}
  \item[\(1\)-parameter families:] \emph{fully classify them up to the weak equivalence}.

    This was accomplished in~\cite{St18,YuINS,GIS-semistable}.
    In particular, generic \(1\)-parameter families of vector fields on the sphere are structurally stable.
  \item[\(2\)-parameter families:] \emph{are they generically structurally stable with respect to the weak and the moderate topological equivalences?}

    We expect that the answer is \enquote{yes}.
    Some two-parameter families were studied in~\cite{Roitenberg-thesis:transl,roit12:trans,Shashkov92,Dukov18-heart:en,Kuznetsov_2004}, but the general question is wide open.
  \item[\(3\)-parameter families:] \emph{list all locally generic structurally unstable families.}

    In this paper we provide two new examples of locally generic structurally unstable \(3\)-parameter families.
    Namely, generic unfoldings of degenerate vector fields in~\autoref{fig:ears-glasses} are structurally unstable; see \autoref{thm:ears-glasses} for a precise statement.
  \item[finite parameter families:] \emph{how far can they fall from being structurally stable?}

    Theorem 2 in~\cite{IKS-th1} provides an example of \(6\)-parameter families with functional invariants.
    This result was independently improved by \citeauthor{Dukov-nonlocal}~\cite{Dukov-nonlocal} and by the first two authors of the present paper~\cite{GK-glasses-modified}, see discussion in \autoref{sec:plans}.
\end{description}

See also~\cite{YuI-surv} for a more detailed survey of current progress on these and other related questions.

\begin{figure}
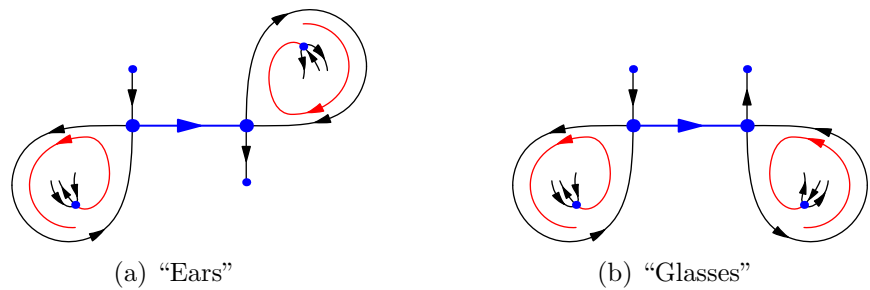

    \centering
    \subcaptionbox{\label{fig:ears}\enquote{Ears}}{\asyinclude{ears}}\hfil
    \subcaptionbox{\label{fig:glasses}\enquote{Glasses}}{\asyinclude{glasses}}
    \caption{\label{fig:ears-glasses}
      Vector fields with \enquote{ears} and \enquote{glasses}
      }
\end{figure}

\section{Preliminaries}%
\label{sec:prelims}
\subsection{Families and equivalences\protect\footnote{Definitions in this section were copied from~\cite{GK-phase,IKS-th1} with minor or no modifications.}}%
\label{sub:equivs}
Denote by \(\Vect\) the Banach space of \(C^3\)-smooth vector fields on the two-sphere.
\begin{definition}
    Given an open subset \({\mathcal B}\subset {\mathbb R}^k\), a~map \(V\colon {\mathcal B}\to \Vect\), \(V=\set{v_{\alpha }}_{\alpha \in {\mathcal B}}\) is called a \emph{\(k\)-parametric family of vector fields}.
    A \emph{local family} is a germ of a map \(V \colon ({\mathbb R}^k, 0)\to (\Vect, v_0)\).
    Denote by \({\mathcal V}_{k}\) the Banach space of local families \(V=\set{v_{\alpha }}_{\alpha \in ({\mathbb R}^{k},0)}\) that are \(C^{3}\)-smooth in \((\alpha , x)\).
\end{definition}

To define equivalent \emph{families} of vector fields, we first introduce an equivalence relation on \(\Vect\).
\begin{definition}%
    \label{def:vf-eq}
    Two vector fields \(v,w\subset  \Vect\) are called \emph{orbitally topologically equivalent}, if there exists an orientation preserving homeomorphism \(H\colon S^2\to S^2\) that takes phase curves of \(v\) to phase curves of \(w\), and preserves time orientation.
\end{definition}

There are several different \enquote{natural} definitions of equivalent families of vector fields.

\begin{definition}%
    \label{def:equiv}
    Consider two local families~\(V=\set{v_{\alpha }}_{\alpha \in ({\mathbb R}^{k},0)}\) and~\({\tilde V}=\set{{\tilde v}_{{\tilde \alpha }}}_{{\tilde \alpha }\in ({\mathbb R}^{k},0)}\).
    A germ of a map
    \begin{align*}
      H&\colon ({\mathbb R}^{k}\times S^2, \set{0}\times S^2)\to ({\mathbb R}^{k}\times S^2, \set{0}\times S^2), & H(\alpha , x)=(h(\alpha ), H_{\alpha }(x))
    \end{align*}
    is called a \emph{weak (topological) equivalence} between~\(V\) and~\({\tilde V}\), if~\(h\colon ({\mathbb R}^{k}, 0)\to ({\mathbb R}^{k}, 0)\) is a germ of a homeomorphism of the parameter spaces, and each~\(H_{\alpha }\) is an orbital topological equivalence between~\(v_{\alpha }\) and~\({\tilde v}_{h(\alpha )}\).

    A weak topological equivalence~\(H\) is called
    \begin{description}
      \item[(strong) topological equivalence,] if \(H\) continuously depends on~\((\alpha , x)\);
      \item[moderate\footnotemark{} topological equivalence,] if\footnotetext{This definition works well only for families such that none of the vector fields have non-hyperbolic singular points. For a more general version, see~\cite{GI-LBS}.} \(H\) is continuous at every point \((0, x)\), where \(x\) is either a singular point of \(v_{0}\), or belongs to the union of all periodic orbits and separatrices of~\(v_{0}\), and \(H^{-1}\) is continuous at every point of a similar set for~\({\tilde V}\);
      \item[weak topological equivalence with \(Sep\)-tracing,] if \(H\) satisfies the following property.
        Let \(S_{\alpha }\), \(\alpha \in ({\mathbb R}^{k}, 0)\) be a continuous family of saddle points of \(v_{\alpha }\), and \((\gamma _{\alpha }, S_{\alpha })\) be a continuous family of local separatrices of these saddles.
        Then \(H_{\alpha }(S_{\alpha })\) and \(H_{\alpha }((\gamma _{\alpha }, S_{\alpha }))\) are continuous families of saddle points and their local separatrices of \({\tilde v}_{h(\alpha )}\).
    \end{description}
\end{definition}
Pros and cons of these and some other equivalence relations are discussed in~\cite{GI-equivs}.
In this paper we will mostly use the last one.
\subsection{Main Theorems}
\subsubsection{Invariant functions and numerical invariants}
Let \({\mathbf{M}}\subset \Vect\) be a Banach submanifold, \(\codim {\mathbf{M}}<\infty \);
let \(k\) be a natural number, \(k\geq \codim {\mathbf{M}}\).
Denote~by~\({\mathbf{M}}^{\pitchfork ,k}\subset {\mathcal V}_{k}\) the set of local families~\(V\) such that \(v_0\in {\mathbf{M}}\) and \(V\) is transverse to~\({\mathbf{M}}\) at~\(v_0\).
All numerical invariants of local families constructed in~\cite{IKS-th1,GK-phase} follow the same pattern:
they have the form \(V\mapsto \varphi (v_{0})\), where \(\varphi \colon {\mathbf{M}}\to {\mathbb R}\) is an \emph{invariant function} in the following sense.
\begin{definition}
    [{cf.~\cite[Definition 16]{IKS-th1}}]
    Let~\(\sim\) be one of the equivalences on~\({\mathcal V}_{k}\) defined above.
    A~function \(\varphi \colon {\mathbf{M}}\to {\mathbb R}\) defined on a Banach submanifold of~\(\Vect\) is called \emph{invariant} with respect to~\(\sim\), if for any two \(\sim\)-equivalent local families \(V, {\tilde V}\in {\mathbf{M}}^{\pitchfork ,\codim {\mathbf{M}}}\) we have \(\varphi (v_{0})=\varphi ({\tilde v}_{0})\).

    A function \(\varphi \colon {\mathbf{M}}\to {\mathbb R}\) is called \emph{robustly invariant} with respect to~\(\sim\), if the same equality holds for any two \(\sim\)-equivalent families \(V, {\tilde V}\in {\mathbf{M}}^{\pitchfork ,k}\), \(k\geq \codim {\mathbf{M}}\).
\end{definition}

By definition, if \(\varphi \) is an invariant function with respect to~\(\sim\), then \(V\mapsto \varphi (v_{0})\) is a numerical invariant of~classification of~local families~\(V\in {\mathbf{M}}^{\pitchfork ,\codim {\mathbf{M}}}\) with respect~to~\(\sim\).

In order to transfer this invariant to an open set in the space of (non-local in parameter) families of vector fields~\(V\colon {\mathcal B}\to \Vect\), we require~\({\mathbf{M}}\) to be \emph{topologically distinguished} in the following sense.
\begin{definition}
    [{cf.~\cite[Definition 16]{IKS-th1}}]%
    \label{def:top-disting}
    We say that a Banach submanifold \({\mathbf{M}}\subset \Vect\) is \emph{topologically distinguished} in its neighbourhood~\({\mathcal U}\supset {\mathbf{M}}\), if two vector fields \(v\in {\mathbf{M}}\) and \(w\in {\mathcal U}\setminus {\mathbf{M}}\) cannot be orbitally topologically equivalent.
\end{definition}

In \autoref{sec:ears-glasses} we shall prove the following theorem.
\begin{theorem}%
    \label{thm:ears-glasses}
    There exist a submanifold \({\mathbf{M}}\subset \Vect\) and a smooth function~\(\varphi \colon {\mathbf{M}}\to {\mathbb R}\) such that
    \begin{itemize}
      \item \({\mathbf{M}}\) is topologically distinguished in its sufficiently small neighbourhood;
      \item \(\varphi \) is robustly invariant with respect to the weak topological equivalence with \(Sep\)-tracing;
      \item the image of \(\varphi \) is the set of all positive numbers;
      \item for all \(v\in {\mathbf{M}}\), \(d\varphi (v)\neq 0\).
    \end{itemize}
\end{theorem}
More precisely, one example of a pair \(({\mathbf{M}}, \varphi )\) with these properties is the submanifold~\(\mathring {\mathbf{T}}\) and the function~\(\nu \) from~\cite[Sec.~2.1.3]{IKS-th1}.
We provide two new examples of pairs \(({\mathbf{M}}, \varphi )\) satisfying the conclusions of this theorem.

As explained above, the second conclusion of \autoref{thm:ears-glasses} implies that \(V\mapsto \varphi (v_{0})\) is an invariant of classification of families \(V\in {\mathbf{M}}^{\pitchfork ,k}\), \(k\geq 3\), up to the weak equivalence with \(Sep\)-tracing.
The last two conclusions show that this invariant is non-degenerate; in particular, all families \(V\in {\mathbf{M}}^{\pitchfork ,k}\) are structurally unstable in~\({\mathbf{M}}^{\pitchfork ,k}\).

In \autoref{thm:ears-glasses} we prove that \(\varphi \) is not just an invariant function, but a \emph{robustly} invariant function for the sake of future applications, see \autoref{sec:plans}.

\subsubsection{Local and non-local families}
Let us explain how the first conclusion of~\autoref{thm:ears-glasses} enables us to transfer this statement to an open set in the space of \emph{non-local} families.
We shall need a definition of the weak equivalence for non-local families, cf.~\autoref{def:equiv}.
\begin{definition}
    Two non-local families \(V\colon {\mathcal B} \to  \Vect\), \({\tilde V}\colon {\tilde {\mathcal B}} \to  \Vect\), \({\mathcal B}, {\tilde {\mathcal B}}\subset {\mathbb R}^{k}\), are said to be \emph{weakly topologically equivalent}, if there exists a homeomorphism~\(h\colon {\mathcal B} \to  {\tilde {\mathcal B}}\) and a family of orientation-preserving homeomorphisms \(H_{\alpha }\colon S^2\to S^2\), \(\alpha \in {\mathcal B}\), such that for each \(\alpha \), the germ of \(H\colon (\alpha , x)\mapsto (h(\alpha ), H_{\alpha }(x))\) at this value of~\(\alpha \) is a weak equivalence between the local families \((V, \alpha )\) and \(({\tilde V}, h(\alpha ))\).
\end{definition}
Other equivalences can be similarly transferred to the class of non-local families.

Let \({\mathbf{M}}\subset \Vect\) be a Banach submanifold that is topologically distinguished in its neighbourhood~\({\mathcal U}\).
Suppose that a function~\(\varphi \colon {\mathbf{M}}\to {\mathbb R}\) is invariant with respect to \(\sim\).
Consider the space~\({\mathbf{M}}^{\pitchfork }_{nonloc.}\) of \(k\)-parameter non-local families~\(V=\set{v_{\alpha }}_{\alpha \in {\mathcal B}}\), \(k=\codim {\mathbf{M}}\), with the following properties.
\begin{itemize}
  \item \(v_{\alpha }\in {\mathcal U}\) for all \(\alpha \in {\mathcal B}\);
  \item \(V\) meets~\({\mathbf{M}}\) at a single vector field~\(v\);
  \item \(V\) is transverse to~\({\mathbf{M}}\) at~\(v\).
\end{itemize}
Since \({\mathbf{M}}\) is topologically distinguished, the map \({\mathbf{M}}^{\pitchfork }_{nonloc.}\to {\mathbf{M}}^{\pitchfork ,k}\) given by \(V\mapsto (V, V\cap {\mathbf{M}})\) sends \(\sim\)-equivalent non-local families to \(\sim\)-equivalent local families.
Therefore, the formula \(V\mapsto \varphi (V\cap {\mathbf{M}})\) defines a numerical invariant of classification of non-local families~\(V\in {\mathbf{M}}^{\pitchfork }_{nonloc.}\) up to the \(\sim\)-equivalence.

Applying these arguments to \autoref{thm:ears-glasses}, one can prove the following theorem.
\begin{theorem}%
    \label{thm:ears-glasses-nonloc}
    There exists an open subset in the space of non-local \(3\)-parameter families of vector fields on the sphere such that the classification of these families up to the weak topological equivalence with \(Sep\)-tracing has a numerical invariant.
    In particular, families from this open set are not structurally stable.
\end{theorem}

\section{Saddle loop bifurcation}\label{sec:saddle-loop}
\subsection{Statement of the theorem and corollaries}\label{sub:sl:statement}
Recall that for a hyperbolic saddle point~\(L\) of a vector field~\(v\in \Vect\), its \emph{characteristic number} is the absolute value of the ratio of the eigenvalues of the linearization of~\(v\) at~\(L\), the negative one is in the numerator.

\begin{figure}
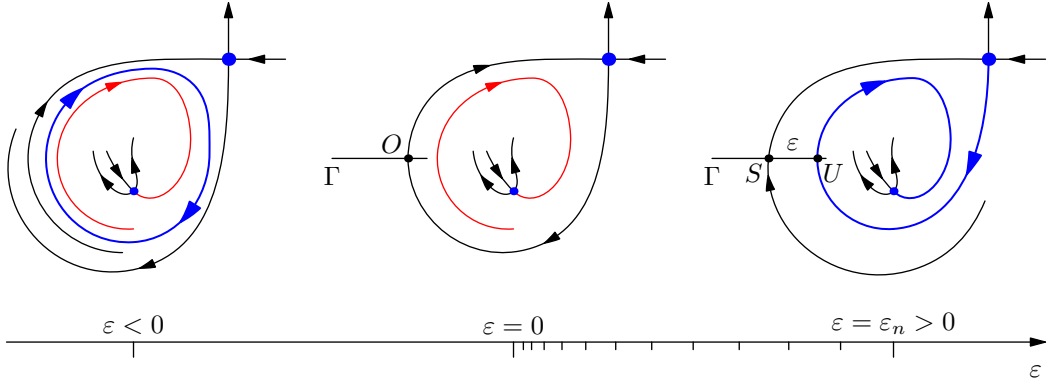

    \centering
    \asyinclude{saddle-loop-bifurcation}
    \caption{Saddle loop bifurcation}\label{fig:saddle-loop}
\end{figure}

Consider a~vector field~\(v_0\in \Vect\) with a hyperbolic saddle~\(L\).
Suppose that
\begin{itemize}
  \item the characteristic number~\(\lambda \) of~\(L\) is greater than one, \(\lambda >1\);
  \item an~unstable separatrix~\(l^{u}\) of~\(L\) forms a~separatrix loop~\(l\) with a~stable separatrix~\(l^{s}\) of~the same saddle;
  \item an~unstable separatrix~\(\gamma \) of~another hyperbolic saddle~\(I\) winds onto~\(l\).
\end{itemize}

Given a~vector field~\(v\in (\Vect, v_0)\), let \(L(v)\) and \(I(v)\) be the hyperbolic saddles of~\(v\) close to \(L\) and \(I\);
let~\(l^{u}(v)\), \(l^{s}(v)\), and~\(\gamma (v)\) be separatrices of these saddles such that the germs \((l^{u}(v), L(v))\), \((l^{s}(v), L(v))\), and \((\gamma (v), I(v))\) are close to the germs \((l^{u}, L)\), \((l^{s}, L)\), and \((\gamma , I)\), respectively.
Fix a cross-section \((\Gamma , O)\) to~\(l\) and a coordinate \(x\colon (\Gamma , O)\to ({\mathbb R}, 0)\) so that \(x\) is positive on \(\Gamma \cap \gamma \);
let \(S(v)\) and \(U(v)\) be the first intersection points of \(l^{s}(v)\) and \(l^{u}(v)\) with this cross-section, counting from~\(L\).
The difference
\[
    \eps (v)=x(S(v))-x(U(v))
\]
is called the \emph{separatrix splitting parameter} for the separatrix loop~\(l\), see \autoref{fig:saddle-loop}.

Denote by \(({\mathbf{S}}{\mathbf{L}}_{l}, v_0)\subset (\Vect, v_0)\) the germ of the codimension-one Banach submanifold of~\(\Vect\) given by \(\eps (v)=0\).
Geometrically, this equation means that \(l^{u}(v)\) coalesces with~\(l^{s}(v)\), so the separatrix loop~\(l\) survives.

Consider an unfolding \(V=\set{v_{\alpha }}_{\alpha \in ({\mathbb R}^{k},0)}\), \(k\geq 1\), of~\(v_0\) transverse to~\(({\mathbf{S}}{\mathbf{L}}_{l}, v_0)\).
Let us reparametrize this family so that \(\alpha =(\eps , \beta )\).
The objects introduced above (\(L\), \(I\), \(l^{s}\), \(l^{u}\), \(S\), \(U\), \(\eps \)) can be considered as functions of~\(\alpha \): \(\eps (\alpha )\coloneqq \eps (v_{\alpha })\) etc.

Consider another vector field \({\tilde v}_0\) of the same type, its saddle loop \({\tilde l}\), and an unfolding \({\tilde V}\) of \({\tilde v}_0\) transverse to \({\mathbf{S}}{\mathbf{L}}_{{\tilde l}}\).
Denote by \({\tilde L}\) etc.\ the objects that play the same role for~\({\tilde V}\) as the corresponding objects for~\(V\).

The following theorem reinterprets results of~\cite{IKS-th1} as a reusable statement incapsulating usage of sparkling separatrix connections.

\begin{theorem}%
    \label{thm:ln2-eps-ratio}
    In the settings introduced above, suppose that \(H\colon (\alpha , x)\mapsto (h(\alpha ), H_{\alpha }(x))\) is a weak equivalence between \(V\) and \({\tilde V}\) such that \(H_{\alpha }\) sends \(L(\alpha )\), \(l^{s}(\alpha )\), \(l^{u}(\alpha )\), \(I(\alpha )\), and \(\gamma (\alpha )\) to \({\tilde L}(h(\alpha ))\), \({\tilde l}^{s}(h(\alpha ))\), \({\tilde l}^{u}(h(\alpha ))\), \({\tilde I}(h(\alpha ))\), and \({\tilde \gamma }(h(\alpha ))\), respectively.
    Then for \(\eps >0\), \(({\tilde \eps }, {\tilde \beta })=h(\eps , \beta )\) the difference
    \begin{equation}
        \label{eq:ln2-difference}
        \frac{\ln(-\ln {\tilde \eps })}{\ln {\tilde \lambda }(0, {\tilde \beta })}-\frac{\ln(-\ln \eps )}{\ln \lambda (0, \beta )}
    \end{equation}
    is uniformly bounded in \(\set{(\eps , \beta ) | 0<\eps <C, \|\beta \|<C}\) for some \(C>0\), and
    \begin{equation}
        \label{eq:ln2-eps-ratio}
        \lim_{\substack{\eps \to 0+\\\beta \to 0}}\frac{\ln(-\ln {\tilde \eps })}{\ln(-\ln \eps )}=\frac{\ln {\tilde \lambda }(0)}{\ln \lambda (0)}.
    \end{equation}
\end{theorem}
Geometrically, \(\frac{\ln(-\ln \eps )}{\ln \lambda (0, \beta )}\) is the number of turns \(\gamma \) makes around~\(l\) before coming to the interval \([S({\tilde v}), U({\tilde v}))\subset \Gamma \).

Before we prove this theorem, let us use it to prove two corollaries.
\begin{corollary}
    [cf.~\cite{DR90}]%
    \label{cor:lambda-eq}
    In the settings of \autoref{thm:ln2-eps-ratio}, suppose that both \(h\) and \(h^{-1}\) are Hölder continuous.
    Then \(\lambda (0)={\tilde \lambda }(0)\).
\end{corollary}
\begin{proof}
    Choose \(C\) and \(\kappa \) such that both \(h\) and \(h^{-1}\) are Hölder continuous with Hölder coefficient~\(C\) and Hölder exponent~\(\kappa \).
    Then for \(({\tilde \eps }, {\tilde \beta })=h(\eps , \beta )\) we have
    \[
        |{\tilde \eps }| \leq  \left\|({\tilde \eps }, {\tilde \beta }) - h(0, \beta )\right\| \leq  C \|(\eps , \beta )-(0, \beta )\|^{\kappa }=C|\eps |^{\kappa },
    \]
    and similarly \(|\eps |\leq C|{\tilde \eps }|^{\kappa }\).
    These inequalities imply that the left hand side of~\eqref{eq:ln2-eps-ratio} equals one, thus \(\lambda (0, 0)={\tilde \lambda }(0, 0)\).
\end{proof}
\begin{corollary}
    Let \({\mathbf{S}}{\mathbf{L}}\subset \Vect\) be the Banach manifold of vector fields \(v\in \Vect\) of the type described above \emph{with no other degeneracies}.
    Then \(\lambda \colon {\mathbf{S}}{\mathbf{L}}\to {\mathbb R}_+\) is a robustly invariant function of classification of unfoldings \(V\in {\mathbf{S}}{\mathbf{L}}^{\pitchfork ,k}\), \(k\geq 1\), up to the weak equivalence with \(Sep\)-tracing and additional restriction \enquote{both \(h\) and \(h^{-1}\) are Hölder continuous}.
\end{corollary}
This corollary immediately follows from \autoref{cor:lambda-eq}.
\begin{remark}
    \autoref{thm:ln2-eps-ratio} and the corollaries can be easily generalized to the case of vector fields with a monodromic hyperbolic polycycle~\(\Pi \), with a few modifications:
    \begin{itemize}
      \item \(l\) is a separatrix connection of~\(\Pi \), not necessarily a separatrix loop;
      \item \(\eps \) and \({\tilde \eps }\) are separatrix splitting parameters corresponding to~\(l\) and~\({\tilde l}=H_0(l)\);
      \item \eqref{eq:ln2-eps-ratio} holds on the subspace given by the condition \enquote{all other separatrix connections of~\(\Pi \) survive};
      \item \(H\) is a moderate equivalence between \(V\) and \({\tilde V}\);
        otherwise \(h\) may a priori send the subspace described above of~\(V\) to another subspace of the parameter space of~\({\tilde V}\).
    \end{itemize}
\end{remark}

In the rest of this section we prove \autoref{thm:ln2-eps-ratio}.
\subsection{Sparkling separatrix connections}
Consider a~vector field~\(v\in (\Vect, v_0)\).
If \(\eps (v)=0\), then the separatrix loop~\(l\) survives.
If \(\eps (v)<0\), then an attracting hyperbolic limit cycle appears near~\(l\), and this limit cycle separates~\(l^{u,s}(v)\) from~\(\gamma (v)\).

In the last case~\(\eps (v)>0\), the separatrix~\(\gamma (v)\) makes several turns around~\(l\), then comes to~the interval~\({[S(v), U(v))}\).
The coordinate change by an appropriate Dehn twist changes the number of turns~\(\gamma (v)\) makes around~\(l\) by one, so this number~\(N(v)\) is not a well-defined function of a vector field \(v\in (\Vect, v_0)\).
However, we can define~\(N(v)\) \emph{up to an additive constant}, cf.~\cites[Definition 13]{IKS-th1}[Remark 2]{GK-phase}.
Namely, we fix an intersection point \(P\in \gamma \cap \Gamma \), then choose a continuous family of points \(P(v)\in \gamma (v)\cap \Gamma \), \(P(v_0)=P\), and let \(N(v)\) be the cardinality of \(\gamma (v)\cap (U(v), P(v)]\subset \gamma (v)\cap \Gamma \).
Clearly, a different choice of~\(\Gamma \) and/or~\(P\) leads to a function of the form \(N(v)+\const\), hence \(N\) is defined up to an additive constant.

From now on, we fix~\(\Gamma \) and~\(P\), hence the function~\(N\).
This function has step discontinuities at vector fields~\(v\) such that \(S(v)\in \gamma (v)\).
This means that~\(\gamma (v)\) makes \(N(v)\)~turns around~\(l\), then coalesces with~\(l^{s}(v)\).

Given a local family \(V\) as in \autoref{thm:ln2-eps-ratio}, denote by \({\mathcal C}\) the set of values of \(\alpha \) such that \(v_{\alpha }\) has a separatrix connection described in the previous paragraph.
This set is a union of connected components \({\mathcal C}_{n}\) enumerated by \(N(\alpha )\coloneqq N(v(\alpha ))\).

Due~to~\cite[Lemma 4]{IKS-th1}, in some small neighborhood of the origin, each \({\mathcal C}_{n}\) is the graph of a function \(\eps =\eps _{n}(\beta )\).
In other words, for each \(\beta \) from a small neighborhood of the origin and \(n\) large enough, there exists a unique \(\eps =\eps _{n}(\beta )\) such that \(v_{\alpha }=v_{\eps , \beta }\) has a separatrix connection between \(\gamma (\alpha )\) and \(l^{s}(\alpha )\) with \(n\)~turns.
Moreover, these functions form a decreasing functional sequence such that
\begin{equation}%
    \label{eq:ln2-en}
    \ln(-\ln \eps _n(\beta )) = n\ln \lambda (0, \beta ) + O(1),
\end{equation}
see also~\cite[Lemma 6]{GK-phase} for a more precise estimate.
Here \(O(1)\) term in the right hand side of~\eqref{eq:ln2-en} is uniformly bounded in some neighborhood \(\set{(\eps , \beta ) | 0<\eps <C, \|\beta \|<C}\).

Note that \(N(\alpha )\), \(\alpha =(\eps , \beta )\), can be equivalently defined by the inequalities
\[
    \eps _{N(\alpha )+1}(\beta )\leq \eps <\eps _{N(\alpha )}(\beta ).
\]
Then~\eqref{eq:ln2-en} implies that
\begin{align}
  \label{eq:ln2-en-N}
  N(\eps , \beta ) &= \frac{\ln(-\ln \eps )}{\ln \lambda (0, \beta )} + O(1) &\text{as }\eps &\to 0+ &\text{uniformly in }\beta .
\end{align}

In the next section we shall use this fact to prove \autoref{thm:ln2-eps-ratio}.

\subsection{Comparing two families}%
\label{sec:comp-two-famil}
Consider two equivalent families \(V\) and \({\tilde V}\) as in \autoref{thm:ln2-eps-ratio}.
Let \({\mathcal C}\) be the set defined above, and \({\tilde {\mathcal C}}\) be the similar set for~\({\tilde V}\).
Since \(H_{\alpha }\) sends \(\gamma (v)\) and \(l^{s}(v)\) to \({\tilde \gamma }({\tilde v})\) and \({\tilde l}^{s}({\tilde v})\), we have \(h({\mathcal C})={\tilde {\mathcal C}}\).
Since \(h\) is a homeomorphism, it sends each connected component \({\mathcal C}_{n}\) to a connected component \({\tilde {\mathcal C}}_{n'}\), possibly with a different index.
However, it preserves the relative order of these connected components, hence there exists a constant \(a\in {\mathbb Z}\) such that for \(n\) large enough we have \(h({\mathcal C}_{n})={\tilde {\mathcal C}}_{n+a}\).
Recall that the function \(N\) defined above is defined up to an additive constant, so we may and will assume that~\({\tilde N}(h(\alpha ))=N(\alpha )\).

Now take \(({\tilde \eps }, {\tilde \beta })=h(\eps , \beta )\), and substitute asymptotic estimates~\eqref{eq:ln2-en-N} both for \(N(\eps , \beta )\) and~\({\tilde N}({\tilde \eps }, {\tilde \beta })\) in the formula \({\tilde N}({\tilde \eps }, {\tilde \beta })=N(\eps , \beta )\).
This immediately implies that the difference~~\eqref{eq:ln2-difference} is uniformly bounded in some neighborhood \(\set{(\eps , \beta )| 0<\eps <C, \|\beta \|<C}\).

In order to prove~\eqref{eq:ln2-eps-ratio}, it suffices to apply the estimate from the previous paragraph, and use the fact that \(\ln(-\ln \eps )\to \infty \) as~\(\eps \to 0\).

\section{Vector fields with \enquote{ears} or \enquote{glasses}}%
\label{sec:ears-glasses}
In this section we prove \autoref{thm:ears-glasses}.
In \autoref{sec:eg:descr}, we describe~\({\mathbf{M}}\) and~\(\varphi \).
In \autoref{sub:ssc} we consider two equivalent families \(V, {\tilde V}\in {\mathbf{M}}^{\pitchfork ,k}\), \(k\geq 3\), and use \autoref{thm:ln2-eps-ratio} to deduce an asymptotic relation on the components of~\(\alpha \) and~\(h(\alpha )\).
Finally, in \autoref{sub:synchr-subf} we describe a special \enquote{synchronizing} subfamily;
restriction of the relation from \autoref{sub:ssc} to this subfamily implies \(\varphi (v_0)=\varphi ({\tilde v}_0)\).

\subsection{Special classes of degenerate vector fields}%
\label{sec:eg:descr}
The manifold \({\mathbf{M}}\) consists of two very similar disjoint components \({\mathbf{E}}\) and \({\mathbf{G}}\) (from \enquote{ears} and \enquote{glasses}), \({\mathbf{M}}={\mathbf{E}}\sqcup {\mathbf{G}}\), see \autoref{fig:ears-glasses}.
We shall describe these components simultaneously.

\subsubsection{Separatrix graphs \enquote{ears} and \enquote{glasses}}%
Consider a vector field \(v\).
Suppose that it has two hyperbolic saddles \(L\) and \(R\) (from \enquote{left} and \enquote{right}), and the following separatrix connections:
\begin{itemize}
  \item saddles \(L\) and \(R\) have separatrix loops \(l\) and \(r\), respectively;
  \item the \enquote{unused} unstable separatrix of \(L\) coalesces with the \enquote{unused} stable separatrix of \(R\), forming a separatrix connection \(b\) (from \enquote{bridge}).
\end{itemize}

The loops~\(l\) and~\(r\) split the sphere into two discs and one annulus.
Mark one of the points in the annulus as \enquote{infinity}, \(\infty \notin b\), and make a stereographic projection \(S^2\setminus \set{\infty }\to {\mathbb R}^2\).
Then we can talk about various points being \enquote{inside} or \enquote{outside} some curves.

In particular, orientation of the loops \(l\), \(r\) (\enquote{clockwise} or \enquote{counter-clockwise}) is well-defined.
If a vector field has saddles and separatrix connections described above, we say that it has \emph{a separatrix graph \enquote{ears} or \enquote{glasses}} depending on the orientation of \(l\) and \(r\), namely
\begin{description}
  \item[{for \enquote{glasses},}] the loops \(l\) and \(r\) are oriented in \emph{the same way};
  \item[{for \enquote{ears},}] the loops \(l\) and \(r\) are oriented in \emph{the opposite ways}.
\end{description}

\subsubsection{Winding separatrices}%
\label{sec:winding-seps}
Assume that the characteristic numbers \(\lambda \) and \(\rho \) of \(L\) and \(R\) satisfy the inequalities
\begin{align}
  \label{eq:charnum:ineq}
  \lambda &>1, & \rho &<1.
\end{align}
These inequalities imply that the loop \(l\) attracts from the inside, while the loop \(r\) repels from the inside, see~\cite[Remark 12]{IKS-th1}.
Next, assume that there is a hyperbolic saddle point \(I_L\) inside \(l\), and one of the unstable separatrices \(\gamma _L\) of \(I_L\) winds onto \(l\).
Formally, the \(\omega \)-limit set of \(\gamma _L\) is \(l\cup \set{L}\).
Similarly, there is a hyperbolic saddle \(I_{R}\) inside~\(r\), and one of its stable separatrices~\(\gamma _R\) winds onto~\(r\) in the reverse time.

This completes the description of the \enquote{interesting} part of the phase portrait.
For technical reasons, we also require that the rest of the phase portrait is \enquote{not interesting}, i.e.\ structurally stable.
\begin{definition}%
    \label{def:ears-glasses}
    We say that a vector field \(v\) belongs to~\({\mathbf{E}}\) (resp. \({\mathbf{G}}\)), if
    \begin{enumerate}
      \item it has a separatrix graph \enquote{ears} (resp. \enquote{glasses});
      \item the characteristic numbers \(\lambda \) and \(\rho \) of \(L\) and \(R\) satisfy the inequalities~\eqref{eq:charnum:ineq};
      \item it has hyperbolic saddle points~\(I_{L}\), \(I_{R}\), and their separatrices~\(\gamma _{L}\) and~\(\gamma _{R}\) winding onto~\(l\) and~\(r\) as described above;
      \item\label{it:absorbing} the unused stable separatrix of~\(L\) tends to a hyperbolic source or a hyperbolic repelling cycle in the reverse time, and the unused unstable separatrix of~\(R\) tends to a hyperbolic sink or a hyperbolic attracting cycle in the forward time;
      \item the restriction of \(v\) to the complement of a small neighborhood of \(l\cup b\cup r\) is structurally stable.
    \end{enumerate}
    For a vector field \(v\in {\mathbf{E}}\cup {\mathbf{G}}\), we put
    \begin{equation}
        \label{eq:eg:phi}
        \varphi (v)=-\frac{\ln \rho (v)}{\ln \lambda (v)}.
    \end{equation}
\end{definition}

Recall that \({\mathbf{M}}={\mathbf{E}}\sqcup {\mathbf{G}}\).
It is easy to see that \({\mathbf{M}}\) is an embedded Banach submanifold of \(\Vect\) of codimension \(3\), i.e.\ near each vector field \(v\in {\mathbf{M}}\) there exists a germ of a smooth map \(\psi \colon (\Vect, v)\to ({\mathbb R}^3, 0)\) of full rank such that \(({\mathbf{M}}, v)=(\psi ^{-1}(0), v)\).
Indeed, it suffices to take \(\psi =(\eps , \sigma , \delta )\), where \(\eps \), \(\sigma \), and \(\delta \) are the separatrix splitting parameters for \(l\), \(b\), \(r\), respectively, defined as in \autoref{sub:sl:statement}.
Note that a similar statement is false at some points of the \emph{closure} of \({\mathbf{M}}\) in \(\Vect\), e.g.\ those corresponding to vector fields with two \enquote{ears} graphs at different locations.

The following theorem is an explicit version of \autoref{thm:ears-glasses}.
\begin{theorem}%
    \label{thm:eg:expl}
    The function~\(\varphi \colon {\mathbf{M}}\to {\mathbb R}\) defined above satisfies all conclusions of~\autoref{thm:ears-glasses}.
\end{theorem}

Clearly, \(\varphi ({\mathbf{M}})\) is the set of positive numbers, and \(d\varphi \) does not vanish anywhere on \({\mathbf{M}}\).
The last two requirements of \autoref{def:ears-glasses} guarantee that \({\mathbf{M}}\) is topologically distinguished in its sufficiently small neighborhood.
So, it remains to show that \(\varphi \) is robustly invariant with respect to the weak topological equivalence with \(Sep\)-tracing.

\subsubsection{Technical hypotheses}%
\label{ssub:eg:technical}
We shall prove \autoref{thm:eg:expl} in slightly more general settings.
Namely, we want to isolate parts of the proof that use two technical hypotheses: the \(Sep\)-tracing property and structural stability of \(v\) away from the separatrix graph \enquote{ears} or \enquote{glasses}, see \autoref{def:ears-glasses}.

In order to get rid of these two requirements, we consider vector fields with \emph{marked} saddles \(L\), \(R\), \(I_{L}\), \(I_{R}\) and separatrices \(\gamma _{L}\), \(\gamma _{R}\), and require that \(H_{\alpha }\) sends these objects for \(v_{\alpha }\) to the corresponding objects for \({\tilde v}_{h(\alpha )}\).

Namely, consider a vector field \(v\), its hyperbolic saddles \(L\), \(R\), \(I_{L}\), \(I_{R}\), and separatrices \(l\), \(b\), \(r\), \(\gamma _{L}\), \(\gamma _{R}\) that satisfy the first four requirements of \autoref{def:ears-glasses}.
For a small perturbation of~\(v\), condition \enquote{separatrix connections \(l\), \(b\), \(r\) survive} defines a germ of a Banach submanifold \(({\mathbf{M}}_{L,R}, v)\subset (\Vect, v)\) of codimension~\(3\), and~\eqref{eq:eg:phi} defines a germ of a smooth function \(\varphi _{L,R}\colon({\mathbf{M}}_{L,R},v)\to ({\mathbb R},\varphi _{L,R}(v))\).
In the case \(v\in {\mathbf{M}}\), these germs concide with \(({\mathbf{M}}, v)\) and \(\varphi |_{({\mathbf{M}},v)}\), respectively.

Let \(V=\set{v_{\alpha }}_{\alpha \in ({\mathbb R}^{k},0)}\in {\mathbf{M}}_{L,R}^{\pitchfork ,k}\), \(k\geq 3\), be an unfolding of~\(v=v_{0}\) transverse to \({\mathbf{M}}_{L,R}\).
A vector field \(v_{\alpha }\) has saddle points \(L(\alpha )\), \(R(\alpha )\), \(I_{L}(\alpha )\), \(I_{R}(\alpha )\) close to the saddles \(L\), \(R\), \(I_{L}\), \(I_{R}\) of \(v=v_{0}\), and separatrices \(\gamma _{L}(\alpha )\), \(\gamma _{R}(\alpha )\) close to corresponding separatrices of~\(v\).
Separatrix connections \(l\), \(b\), \(r\) of \(v_{0}\) are possibly destroyed, and each of them generates two continuous families of separatrices, one stable and one unstable.
Denote by \(l^{s}(\alpha )\), \(l^{u}(\alpha )\), \(b^{s}(\alpha )\), \(b^{u}(\alpha )\), \(r^{s}(\alpha )\), \(r^{u}(\alpha )\) these families, where superscript denotes (un)stability of the separatrix.

Let \({\tilde V}\) be another family of vector field with marked saddles and separatrices of the same type.
By letters with tilde above we denote objects for \({\tilde V}\) similar to those objects for \(V\) denoted by the same letter without tilde.

The following theorem is a slightly generalized version of \autoref{thm:eg:expl} that does not rely on the two technical assumptions mentioned above.
\begin{theorem}%
    \label{thm:eg:weak}
    In the settings introduced above, suppose that \(H\colon(\alpha ,x)\mapsto (h(\alpha ),H_{\alpha }(x))\) is a weak topological equivalence between \(V\) and \({\tilde V}\) such that \(H_{\alpha }\) sends \(L(\alpha )\), \(R(\alpha )\), \(l^{s,u}(\alpha )\), \(b^{s,u}(\alpha )\), \(r^{s,u}(\alpha )\), \(\gamma _{L}(\alpha )\), \(\gamma _{R}(\alpha )\) to \({\tilde L}(h(\alpha ))\), \({\tilde R}(h(\alpha ))\) etc.
    Then~\(\varphi _{L,R}(v_0)=\varphi _{{\tilde L},{\tilde R}}({\tilde v}_0)\).
\end{theorem}

Let us show that this theorem implies \autoref{thm:eg:expl}.
Consider two families \(V, {\tilde V}\in {\mathbf{M}}^{\pitchfork ,k}\), \(k\geq 3\).
Suppose that they are weakly topologically equivalent with \(Sep\)-tracing property.
Fix all saddles and separatrices from \autoref{sec:winding-seps} for \(V\), and introduce families \(L(\alpha )\) etc. as above.
Due to the \(Sep\)-tracing property, the images of these families under \(H_{\alpha }\) are families with similar properties for \({\tilde V}\).
Thus we can apply \autoref{thm:eg:weak} to these two families, and the assumptions \(H_{\alpha }(L(\alpha ))={\tilde L}(h(\alpha ))\) etc.\ will hold automatically.

Finally, due to the last assumption of \autoref{def:ears-glasses}, \({\tilde v}_{0}\) has a unique separatrix graph \enquote{ears} or \enquote{glasses}, hence \({\tilde L}\) and \({\tilde R}\) are the saddles from \autoref{def:ears-glasses} for \({\tilde v}_{0}\), thus conclusion of \autoref{thm:eg:weak} matches the only non-trivial part of the conclusion of \autoref{thm:eg:expl}.

\begin{remark}
    \autoref{thm:eg:weak} compared to \autoref{thm:eg:expl} replaces a pair of technical assumptions with another one, namely \(H_{\alpha }(L(\alpha ))={\tilde L}(h(\alpha ))\) etc.
    Another approach would be to impose more assumptions on the families \(V\), \({\tilde V}\) while still working with the weak equivalence.
    E.g., one can \enquote{tag} all hyperbolic sinks and sources by surrounding each of them with some number of hyperbolic cycles.
    If the number of cycles is different for each \enquote{nest}, then \(H_{\alpha }\) has to send each nest of \(v_{\alpha }\) to the corresponding nest of~\({\tilde v}_{h(\alpha )}\), and all we have to do is to \enquote{identify} saddle points and their separatrices.
    In some cases, it is possible to identify them based on the \(\alpha \)-limit and \(\omega \)-limit sets of the separatrices; in other cases it might need more subtle topological arguments.
\end{remark}

\subsection{Sparkling separatrix connections}\label{sub:ssc}
Consider two families \(V\), \({\tilde V}\) with marked saddles and separatrices as described in \autoref{ssub:eg:technical}.
Let~\(H\) be a weak equivalence between \(V\) and~\({\tilde V}\) that satisfies the assumptions of \autoref{thm:eg:weak}.

In order to use the presence of the separatrices~\(\gamma _{L}\), \(\gamma _{R}\), and sparkling separatrix connections that appear when we destroy the separatrix loops, we apply \autoref{thm:ln2-eps-ratio} to \(V\), \({\tilde V}\), and~\(H\) twice.
First, we apply this theorem to the separatrix loops \(l\) and \({\tilde l}\), and get
\begin{equation}
    \label{eq:ln2-eps}
    \frac{\ln(-\ln {\tilde \eps })}{\ln(-\ln \eps )}\to \frac{\ln {\tilde \lambda }(0)}{\ln \lambda (0)}
\end{equation}
as \(\eps \to 0+\), \((\sigma , \delta , \eta )\to 0\), \(({\tilde \eps }, {\tilde \sigma }, {\tilde \delta }, {\tilde \eta })=h(\eps , \sigma , \delta , \eta )\).

Second, we want to apply the same theorem to \(r\) and \({\tilde r}\).
Since \(\rho <1\) and \(\gamma _{R}\) winds onto \(r\) in the reverse time, this theorem does not apply literally, so we first reverse time by replacing \(v_{\alpha }\) with \(-v_{\alpha }\), and get
\begin{equation}
    \label{eq:ln2-delta}
    \frac{\ln(-\ln {\tilde \delta })}{\ln(-\ln \delta )}\to \frac{\ln {\tilde \rho }(0)^{-1}}{\ln \rho (0)^{-1}}
\end{equation}
as \(\delta \to 0+\), \((\eps , \sigma , \eta )\to 0\), \(({\tilde \eps }, {\tilde \sigma }, {\tilde \delta }, {\tilde \eta })=h(\eps , \sigma , \delta , \eta )\).

Dividing~\eqref{eq:ln2-eps} by~\eqref{eq:ln2-delta}, we obtain
\[
    \frac{\ln(-\ln {\tilde \eps })}{\ln(-\ln {\tilde \delta })}\div \frac{\ln(-\ln \eps )}{\ln(-\ln \delta )}\to 
    \frac{\ln \rho (0)^{-1}}{\ln \lambda (0)}\div \frac{\ln {\tilde \rho }(0)^{-1}}{\ln {\tilde \lambda }(0)}=\frac{\varphi _{L, R}(v_0)}{\varphi _{{\tilde L}, {\tilde R}}({\tilde v}_0)}
\]
as \(\eps \to 0+\), \(\delta \to 0+\), \(\sigma \to 0\), \(\eta \to 0\), \(({\tilde \eps }, {\tilde \sigma }, {\tilde \delta }, {\tilde \eta })=h(\eps , \sigma , \delta , \eta )\).
In order to complete the proof of \autoref{thm:eg:weak} it suffices to find a~\emph{synchronizing subfamily}~\({\mathcal E}\subset \set{\alpha \in ({\mathbb R}^{k}, 0) | \eps >0, \delta >0}\) such that
\[
    \lim_{\substack{\alpha \to 0\\\alpha \in {\mathcal E}}}\frac{\ln(-\ln \eps )}{\ln(-\ln \delta )}=
    \lim_{\substack{\alpha \to 0\\\alpha \in {\mathcal E}}}\frac{\ln(-\ln {\tilde \eps })}{\ln(-\ln {\tilde \delta })}=1,
\]
where as usual \(({\tilde \eps }, {\tilde \sigma }, {\tilde \delta }, {\tilde \eta })=h(\eps , \sigma , \delta , \eta )\).

\subsection{Synchronizing subfamily}%
\label{sub:synchr-subf}
The synchronizing subfamily~\({\mathcal E}\subset ({\mathbb R}^{k}, 0)\) is defined by the following condition:
\(\alpha \in {\mathcal E}\) if and only if \(v_{\alpha }\) has two different separatrix connections joining~\(L(\alpha )\) to~\(R(\alpha )\).
Due to Assumption~\ref{it:absorbing} of~\autoref{def:ears-glasses}, no orbit can start near~\(R\) and then come to a~ neighbourhood of~\(L\).
Thus there is only one possibility for geometry of these two connections, see \autoref{fig:E}.
In particular, \(\alpha \in {\mathcal E}\) implies \(\eps >0\) and \(\delta >0\).

\autoref{fig:bif-diag} shows parts of the bifurcation diagram of \enquote{glasses} that are relevant to our proof, either directly, or through the proof of \autoref{thm:ln2-eps-ratio}.
In \autoref{fig:bif-diag}, we intentionally omit the curves corresponding to the degeneracies not used in our proof.

\begin{figure}
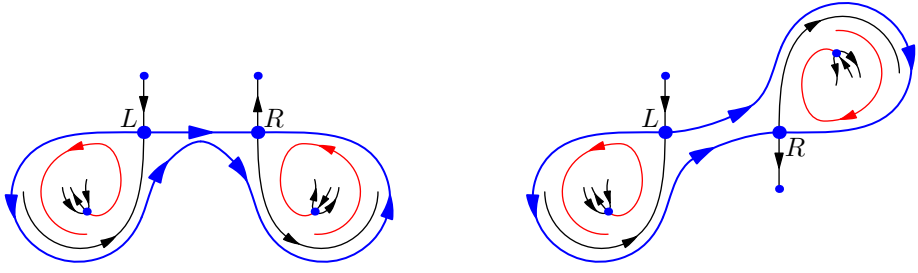

    \centering
    \asyinclude{glasses-unfolded}\hfil
    \asyinclude{ears-unfolded}
    \caption{Vector fields \(v_{\alpha }\), \(\alpha \in {\mathcal E}\), for \enquote{glasses} and \enquote{ears}}\label{fig:E}
\end{figure}

\begin{figure}
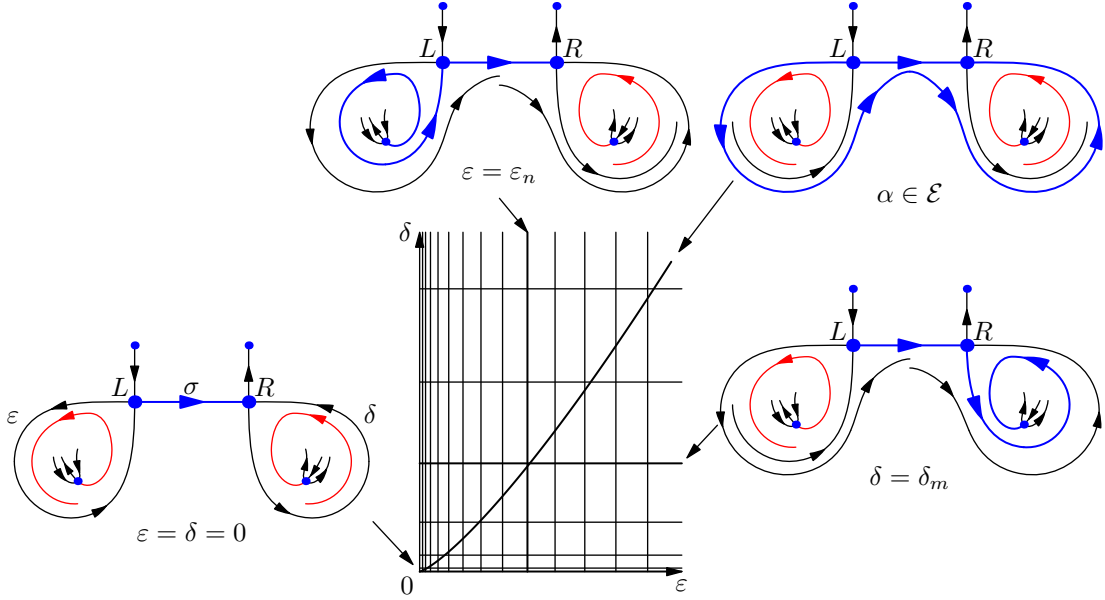

    \centering
    \asyinclude{glasses-bifurcation}
    \caption{Parts of the bifurcation diagram of \enquote{glasses} in the plane \(\sigma =0\) that are relevant to our proof}%
    \label{fig:bif-diag}
\end{figure}

\begin{lemma}%
    \label{lem:E-ln2d-ln2e}
    For the subfamily \({\mathcal E}\) defined above, we have
    \begin{align}
      \label{eq:E-ln2d-ln2e}
      \frac{\ln(-\ln \eps )}{\ln(-\ln \delta )}&\rightrightarrows 1 &\text{as }\eps , \delta &\to 0+, & \alpha &\in {\mathcal E} &\text{uniformly in }&\eta .
    \end{align}
\end{lemma}
We shall prove this lemma separately for \({\mathbf{M}}={\mathbf{G}}\) and for \({\mathbf{M}}={\mathbf{E}}\).
\begin{proof}
    [Proof of \autoref{lem:E-ln2d-ln2e} in the \enquote{glasses} case]
    It is easy to see that one of these connections is a perturbed bridge~\(b\), and the other is close to the union~\(l\cup b\cup r\).
    The former connection exists if and only if~\(\sigma =0\), and the latter connection exists if and only if
    \[
        S_{r}(\alpha )=\Delta (U_{l}(\alpha )),
    \]
    where \(\Delta \) is the correspondence map from the \enquote{outer} half of~\(\Gamma _{l}\) to the \enquote{outer} half of~\(\Gamma _{r}\).
    By definition, \(U_{l}(\alpha )\) is at the distance~\(\eps \) from~\(S_{l}(\alpha )\), and \(S_{r}(\alpha )\) is at the distance~\(\delta \) from~\(U_{r}(\alpha )\), hence in natural coordinates on~\(\Gamma _{l}\), \(\Gamma _{r}\), we have \(\delta =\Delta (\eps )\).

    Recall that the correspondence map of a hyperbolic saddle with characteristic number \(\mu \) is close to \(x\mapsto Cx^\mu \), see e.g.\ \cites[Lemma 5]{IKS-th1}[Lemma 1]{GK07}, and for a chain of maps one should multiply the exponents.
    Therefore, \(\delta =\Delta (\eps )\) implies that
    \begin{equation}
        \label{eq:E-lnd-lne}
        \ln \delta =\lambda (\alpha )\rho (\alpha )\ln \eps +O(1)
    \end{equation}
    as \(\eps \to 0\), \(\delta \to 0\), \(\alpha \in {\mathcal E}\).
    Taking logarithms of both sides, we get~\eqref{eq:E-ln2d-ln2e}.
\end{proof}
\begin{proof}
    [Proof of \autoref{lem:E-ln2d-ln2e} in the \enquote{ears} case]
    In this case, one of the separatrix connections is close to~\(l\cup b\), and the other is close to~\(b\cup r\).
    These connections exist if
    \begin{align*}
      S_{b}(\alpha )&=\Delta _{l}(U_{l}(\alpha )), & S_{r}(\alpha )&=\Delta _{r}(U_{b}(\alpha )),
    \end{align*}
    where \(\Delta _{l}\) is the correspondence map from the \enquote{outer} half of~\(\Gamma _{l}\) to~\(\Gamma _{b}\), and \(\Delta _{r}\) is the correspondence map from one half of~\(\Gamma _{b}\) to~the \enquote{outer} half of~\(\Gamma _{r}\).

    Similarly to the previous case, these equalities imply
    \begin{align*}
      \ln \sigma                      & =\lambda (\alpha )\ln \eps +O(1), & \ln \delta           & =\rho (\alpha )\ln \sigma +O(1).
    \end{align*}
    Hence we have~\eqref{eq:E-lnd-lne}, then complete the proof as in the previous case.
\end{proof}

Finally, recall that \(H_{\alpha }\) sends separatrices of \(L(\alpha )\) and \(R(\alpha )\) to the corresponding separatrices of \({\tilde L}(h(\alpha ))\) and \({\tilde R}(h(\alpha ))\), hence \(h({\mathcal E})={\tilde {\mathcal E}}\).
Thus
\[
    \lim_{\substack{\alpha \to 0\\\alpha \in {\mathcal E}}}\frac{\ln(-\ln {\tilde \eps })}{\ln(-\ln {\tilde \delta })}=
    \lim_{\substack{{\tilde \alpha }\to 0\\{\tilde \alpha }\in {\tilde {\mathcal E}}}}\frac{\ln(-\ln {\tilde \eps })}{\ln(-\ln {\tilde \delta })}=1.
\]

This and the arguments at the end of \autoref{sub:ssc} complete the proof of \autoref{thm:eg:weak}, hence of~\autoref{thm:eg:expl} and of~\autoref{thm:ears-glasses}.
\section{Future plans}\label{sec:plans}
The statement and the main idea of the proof of \autoref{thm:eg:expl} resemble the study of \enquote{tears of the heart} polycycle in~\cite{IKS-th1}.
However, note that for \enquote{ears} and \enquote{glasses}, the separatrix graph is \emph{not} a polycycle:
we can move only from the \enquote{left} half of the picture to the \enquote{right} half, not the other way around, which makes the bifurcation diagram simpler.

The first two authors plan to use this difference to prove the following facts.
\begin{enumerate}
  \item There exists a submanifold \({\mathbf{M}}\subset \Vect\) of codimension~\(3\) such that the classification of local families \(V\in {\mathbf{M}}^{\pitchfork ,3}\) has infinitely many invariants.
  \item For every \(d>0\) there exists a submanifold \({\mathbf{M}}\subset \Vect\) of codimension~\(d+2\) that admits \(d\) independent robustly invariant functions.
  \item For every \(d>0\) there exists a submanifold \({\mathbf{M}}\subset \Vect\) of codimension~\(4\) that admits~\(d\) independent robustly invariant functions.
  \item There exists a submanifold \({\mathbf{M}}\subset \Vect\) of finite codimension that admits infinitely many robustly invariant functions.
\end{enumerate}

Both the second and the third statements imply that there exists an open set in the space of \(5\)-parameter non-local families of vector fields on the sphere such that classification of these families has a \emph{functional} invariant.
This reduces the codimension in~\cite[Theorem 2]{IKS-th1} by one, and for \(d>1\) we get an even better reduction in codimension compared to~\cite[Theorem 3]{IKS-th1}.

Another example of locally generic \(5\)-parameter families with functional invariants was recently found by \citeauthor{Dukov-nonlocal}~\cite{Dukov-nonlocal,DuYuI-nonlocal}.
His example is a modification of the original \enquote{tears of the heart} polycycle.

The last statement from the list implies that vector fields \(v\in {\mathbf{M}}\) have no finite parameter versal deformations, not even with the number of parameters larger than \(\codim {\mathbf{M}}\).
\section*{Acknowledgements}
We are grateful to Yu.\,Ilyashenko for encouraging us to work on global bifurcations of planar vector fields and for valuable discussions.
\printbibliography
\end{document}